\documentclass[10pt]{amsproc}
\usepackage[utf8]{inputenc}
\usepackage[dvipsnames]{xcolor}
\usepackage{todonotes}

\usepackage{etoolbox}
\usepackage{rotating}

\makeatletter
\patchcmd{\@thm}{\let\thm@indent\indent}{\let\thm@indent\noindent}{}{}
\patchcmd{\@thm}{\thm@headfont{\scshape}}{\thm@headfont{\bfseries}}{}{}

\usepackage[T1]{fontenc}

\usepackage{hyperref}

\usepackage{mathtools}

\let\G\undefined

\usepackage{multirow}
\usepackage{longtable}

\usepackage{amsmath}
\usepackage{amsfonts}
\usepackage{amssymb}
\usepackage{float}
\usepackage{amsthm}
\usepackage{comment}
\usepackage{amscd}
\usepackage{amsxtra}
\usepackage{epsfig}
\usepackage{epigraph}
\usepackage{pgfplots}
\usepackage{pdflscape} 
\usepackage{afterpage}

\usepackage{stmaryrd}

\usepackage{listings}
\usepackage{color}

\definecolor{dkgreen}{rgb}{0,0.6,0}
\definecolor{gray}{rgb}{0.5,0.5,0.5}
\definecolor{mauve}{rgb}{0.58,0,0.82}

\usepackage{color}
\usepackage[all]{xy}
\usepackage{verbatim}

\usepackage{eucal}
\usepackage{mathrsfs}

\usepackage{graphicx}
\usepackage{tikz-cd}

\usepackage{url}
\usepackage{verbatim}

\usepackage{xy}
\xyoption{all}

\usepackage{adjustbox}

\usepackage{amsthm}
\usepackage{hhline}
\usepackage{enumitem}

\usepackage{caption} 
\makeatletter
\def\@tocline#1#2#3#4#5#6#7{\relax
  \ifnum #1>\c@tocdepth 
  \else
    \par \addpenalty\@secpenalty\addvspace{#2}%
    \begingroup \hyphenpenalty\@M
    \@ifempty{#4}{%
      \@tempdima\csname r@tocindent\number#1\endcsname\relax
    }{%
      \@tempdima#4\relax
    }%
    \parindent\z@ \leftskip#3\relax \advance\leftskip\@tempdima\relax
    \rightskip\@pnumwidth plus4em \parfillskip-\@pnumwidth
    #5\leavevmode\hskip-\@tempdima
      \ifcase #1
       \or\or \hskip 1em \or \hskip 2em \else \hskip 3em \fi%
      #6\nobreak\relax
    \hfill\hbox to\@pnumwidth{\@tocpagenum{#7}}\par
    \nobreak
    \endgroup
  \fi}
\makeatother

\usepackage[noabbrev,capitalize]{cleveref} 

\crefformat{nul}{(#2#1#3)}
\Crefformat{nul}{(#2#1#3)}

\crefname{section}{\S}{\S\S}
\crefname{subsection}{\S}{\S\S}
\crefname{axioms}{Axiom}{Axioms}
\crefname{exercise}{Exercise}{Exercises}
\crefname{exercisenum}{Exercise}{Exercises}
\crefname{construction}{Construction}{Constructions}
\crefname{problem}{Problem}{Problems}
\crefname{theorem}{Theorem}{Theorems}
\crefname{definition}{Definition}{Definitions}
\crefname{prop}{Proposition}{Propositions}
\crefname{lemma}{Lemma}{Lemmas}
\crefname{example}{Example}{Examples}
\crefname{examplealph}{Example}{Examples}
\crefname{corollary}{Corollary}{Corollaries}
\crefname{nonexample}{Nonexample}{Nonexamples}
\crefname{equation}{}{}
\crefname{summary}{Summary}{Summaries}
\crefname{recollection}{Recollection}{Recollections}
\Crefname{recollection}{Recollection}{Recollections}
\Crefname{nonexample}{Nonexample}{Nonexamples}
\Crefname{corollary}{Corollary}{Corollaries}
\Crefname{corollary}{Corollary}{Corollaries}
\Crefname{axioms}{Axiom}{Axioms}
\Crefname{exercise}{Exercise}{Exercises}
\Crefname{exercisenum}{Exercise}{Exercises}
\Crefname{construction}{Construction}{Constructions}
\Crefname{problem}{Problem}{Problems}
\Crefname{theorem}{Theorem}{Theorems}
\Crefname{definition}{Definition}{Definitions}
\Crefname{prop}{Proposition}{Propositions}
\Crefname{lemma}{Lemma}{Lemmas}
\Crefname{example}{Example}{Examples}
\Crefname{examplealph}{Example}{Examples}
\Crefname{section}{\S}{\S\S}
\Crefname{subsection}{\S}{\S\S}

\DeclareMathOperator{\QCoh}{\mathrm{QCoh}}

\DeclareMathOperator{\Sym}{Sym}

\newtheorem*{conj-moore}{Conjecture~\ref{moore-splitting}}

\theoremstyle{definition}

\newcommand{\LT}{\mathrm{LT}}

\newcommand{\D}{\mathrm{D}}

\newcommand{\SL}{\mathrm{SL}}
\newcommand{\GL}{\mathrm{GL}}

\newcommand{\F}{\mathrm{F}}
\newcommand{\fr}[1]{\mathfrak{#1}}
\newcommand{\g}{\mathfrak{g}}
\newcommand{\G}{\mathrm{G}}

\newcommand{\IndCoh}{\mathrm{IndCoh}}

\newcommand{\reg}{\mathrm{reg}}

\renewcommand{\S}{Section }

\makeatletter
\providecommand{\leftsquigarrow}{%
  \mathrel{\mathpalette\reflect@squig\relax}%
}
\newcommand{\reflect@squig}[2]{%
  \reflectbox{$\m@th#1\rightsquigarrow$}%
}
\makeatother

\renewcommand{\tilde}{\widetilde}

\usepackage{slashed}

\usepackage{relsize}
\usepackage[bbgreekl]{mathbbol}
\usepackage{amsfonts}
\DeclareSymbolFontAlphabet{\mathbb}{AMSb} 
\DeclareSymbolFontAlphabet{\mathbbl}{bbold}

    \usepackage{graphicx}
    \usepackage[english]{babel}
    \usepackage{graphicx}
    \usepackage{framed}
    \usepackage[normalem]{ulem}
    \usepackage{amsmath}

    \usepackage{amsthm}
    \usepackage{amssymb}
    \usepackage{hyperref} 
    \usepackage[capitalize]{cleveref}
    \usepackage{comment}
    \usepackage{tikz}
    \usetikzlibrary{matrix,calc}
    \usepackage{mathtools}
    \usepackage{tikz-cd}
    \usepackage{amsfonts}
    \usepackage{enumerate}
    \usepackage[utf8]{inputenc}
       \usepackage{stmaryrd}
       \usepackage{bbold}
\usepackage[maxbibnames=99, backend=bibtex, style=numeric]{biblatex}
    \renewcommand{\SL}{\operatorname{SL}}

    \newcommand{\LG}{\mathfrak{g}}
    \newcommand{\LH}{\mathfrak{h}}
    \renewcommand{\LT}{\mathfrak{t}}
    \newcommand{\LN}{\mathfrak{n}}
    \newcommand{\LB}{\mathfrak{b}}

    \renewcommand{\\}{\backslash}

    \theoremstyle{definition}
    \newtheorem{Theorem}{Theorem}[section]

    \newtheorem{Corollary}[Theorem]{Corollary}
    \newtheorem{Definition}[Theorem]{Definition}
    \newtheorem{Proposition}[Theorem]{Proposition}
    \newtheorem{Assumption}[Theorem]{Assumption}
    \newtheorem{Remark}[Theorem]{Remark}
    \newtheorem{Example}[Theorem]{Example}
    \newtheorem{Lemma}[Theorem]{Lemma}

\title{The Cotangent Bundle of $G/U_P$ and Kostant-Whittaker Descent}
\author{Tom Gannon}
\thanks{University of California, Los Angeles - gannonth@math.ucla.edu}

    \renewcommand{\S}{\mathcal{S}} 

    \newcommand{\AvN}{\text{Av}_*^N}
    \newcommand{\Avpsi}{\text{Av}_{!}^{\psi}}
    \newcommand{\Symt}{\text{Sym(}\mathfrak{t}\text{)}}

    \newcommand{\LTd}{\LT^{\ast}}

\begin{document}
\renewcommand{\G}{\mathcal{G}}
\newcommand{\AvNTw}{\text{Av}_*^{N, (T, w)}}
\newcommand{\AvGw}{\text{Av}_{\ast}^{G,w}}
\newcommand{\pifin}{\pi_{\text{fin}}}
\newcommand{\pifinL}{\pi_{\text{fin,L}}}

    \newcommand{\ELeftAdjoint}{\text{ev}_{\omega_{\LTd}}}
\newcommand{\ClGlobalDiffOp}{\text{H}^0\Gamma(\mathcal{D}_{G/N})}
\newcommand{\GlobalDiffOp}{\Gamma(\mathcal{D}_{G/N})}
\newcommand{\indsch}{\mathcal{X}}

    \newcommand{\DGCatContk}{\text{DGCat}^k_{\text{cont}}}
\newcommand{\DGCatContL}{\text{DGCat}^L_{\text{cont}}}

\newcommand{\DNTw}{\mathcal{D}(N\backslash G/N)^{T_r,w}}
\newcommand{\DNTwWhit}{\mathcal{D}(N^-_{\psi}\backslash G/N)^{T_r,w}}
\newcommand{\DNWhit}{\mathcal{D}(N^-_{\psi}\backslash G/N)}
\newcommand{\DNTwldeg}{\mathcal{D}(N \backslash G/N)^{T_r,w}_{\text{left-deg}}}
\newcommand{\DNTwnondeg}{\mathcal{D}(N \backslash G/N)^{T_r,w}_{\text{nondeg}}}
\newcommand{\DN}{\mathcal{D}(N\backslash G/N)}
\newcommand{\DNldeg}{\mathcal{D}(N \backslash G/N)_{\text{left-deg}}}
\newcommand{\DNnondeg}{\mathcal{D}(N \backslash G/N)_{\text{nondeg}}}
\newcommand{\DNlambda}{\mathcal{D}^{\lambda}(N \backslash G/B)}
\newcommand{\Dpsilambda}{\mathcal{D}^{\lambda}(N^- _{\psi}\backslash G/B)}
\newcommand{\DbiTw}{\mathcal{D}(N \backslash G/N)^{T \times T, w}}
\newcommand{\DbiTwnondeg}{\mathcal{D}(N \backslash G/N)^{T \times T, w}_{\text{nondeg}}}
\newcommand{\DbiTwnondegheart}{\mathcal{D}(N \backslash G/N)^{(T \times T, w), \heartsuit}_{\text{nondeg}}}
\newcommand{\DbiTwdeg}{\mathcal{D}(N \backslash G/N)^{T \times T, w}_{\text{deg}}}
\newcommand{\DNBlambda}{\mathcal{D}(N \backslash G/_{\lambda}B)}
\newcommand{\DNTwBlambda}{\mathcal{D}(N \backslash G/_{\lambda}B)}
\newcommand{\DWhitBlambda}{\mathcal{D}(N^-_{\psi}\backslash G/_{\lambda}B)}
\newcommand{\HN}{\D(N \backslash G/N)}
\newcommand{\HNTw}{\D(N \backslash G/N)^{T \times T, w}}
\newcommand{\HNTwabbreviated}{\mathcal{H}^{N, (T,w)}}
\renewcommand{\indsch}{\mathcal{X}}
\newcommand{\wdot}{\dot{w}}
\newcommand{\Gaminusalpha}{\mathbb{G}_a^{-\alpha}}
\newcommand{\Aone}{\mathbf{A}}
\newcommand{\Atwo}{\mathcal{L}\text{-mod}(\Aone)}
\newcommand{\algobj}{\mathcal{A}}
\newcommand{\newalgobj}{\mathcal{A}'}
\newcommand{\tilder}{\tilde{r}}
\newcommand{\Ccirc}{\mathring{\C}}
\newcommand{\rootlattice}{\mathbb{Z}\Phi}
\newcommand{\characterlatticeforT}{X^{\bullet}(T)}
\newcommand{\cocharactherlatticeforT}{X_{\bullet}(T)}
\newcommand{\Spec}{\text{Spec}}
\newcommand{\FourierMukai}{\text{FMuk}}
\newcommand{\oneshiftedCartierdual}{c_1}
\newcommand{\quotientmapforcoarsequotient}{\overline{s}}

\newcommand{\tildeV}{\tilde{\mathbb{V}}}
\newcommand{\Vdual}{V^{\vee}}
\newcommand{\generalstacktoGITquotientmap}{\phi}
\newcommand{\SpecofL}{\text{Spec}(L)}
\newcommand{\terminalmapfromC}{\alpha}
\newcommand{\terminalmapfromCmodassociatedstabilizer}{\dot{\alpha}}
\newcommand{\Hpsiliteral}{\mathcal{H}_{\psi}}
\newcommand{\AvNshifted}{\AvN[\text{dim}(N)]}
\newcommand{\hyperplanefixedbys}{V^{\ast}_{s = \text{id}}}
\newcommand{\fieldpossiblydifferentfromgroundfield}{K}
\newcommand{\Avpsishifted}{\Avpsi[-\text{dim}(N)]}
\newcommand{\FI}{F_{I}}
\newcommand{\FD}{F_{\mathcal{D}}}
\newcommand{\LI}{s_*^{\IndCoh}}
\newcommand{\LD}{\Avpsi}
\newcommand{\LIenh}{\pi_*^{\IndCoh, \text{enh}}}
\newcommand{\LDenh}{\text{Av}_!^{\psi, \text{enh}}}

\newcommand{\IndWithoutSignRep}{\text{Ind}(-)^W}
\newcommand{\ResWithoutSignRep}{\text{WRes}}
\newcommand{\IndWithSignRep}{\text{Ind}(- \otimes k_{\text{sign}})^W}
\newcommand{\ResWithSignRep}{\text{WRes}_s}
\newcommand{\parabolicrestrictionLIFTED}{\text{WRes}}
\newcommand{\V}{\mathcal{V}}
\newcommand{\parabolicrestriction}{\text{Res}}

\newcommand{\conjugacyclassofstandardLevis}{\underline{\Theta}}
\newcommand{\moduliOfEvenMonicPolynomialsOfTHISDEGREE}[1]{\mathcal{M}^{#1}}
\newcommand{\moduliOfALLEvenPolynomialsOfTHISDEGREE}[1]{\mathcal{P}^{#1}}
\newcommand{\moduliOfEvenMonicPolynomialsOfdegreeTWOwithMINUSSIGN}{\mathcal{M}_{-}^{2}}

\newcommand{\basedQuasimapsfromProjectiveLinetoAffineClosureofSLTWOwithIsoClassofTHISDEGREE}[1]{\text{Maps}_*^{#1}(\mathbb{P}^1, \overline{\SL_2/N}/T)}
\newcommand{\basedQuasimapsfromProjectiveLinetoAffineClosureofSLTWOwithISOtoTHISDEGREE}[1]{\text{Maps}_*^{#1, \simeq}(\mathbb{P}^1, \overline{\SL_2/N}/T)}

\newcommand{\explicitIsoofBasedQMapsforSLTWOwithPolynomialsforMapsofTHISDEGREE}[1]{\Phi_{#1}}
\newcommand{\basedQuasimapsfromProjectiveLinetoAffineClosureofBORELwithIsoClassofTHISDEGREE}[1]{\text{Maps}_*^{#1}(\mathbb{P}^1, \overline{B/N}/T)}

\newcommand{\basedmapsFromProjectiveLinetoTHISSPACE}[1]{\text{Maps}_*(\mathbb{P}^1, #1)}

\newcommand{\resolutionOfSingularitiesSpaceForBasicAffineSpace}{G\mathop{\times}\limits^{B} E}

\newcommand{\affineClosureofBasicAffineSpace}{\overline{G/N}}
\newcommand{\affineClosureOfCotangentBundleofBasicAffineSpace}{\overline{T^*(G/N)}}
\newcommand{\smoothLocusOfAffineclosureofCotangentBundleOfBasicAffineSpace}{\overline{T^*(G/N)}_{\text{reg}}}
\newcommand{\SpecOfSymofSectionsOfTangentBundleOfBasicAffineSpace}{\text{Spec}(\text{Sym}_A^{\bullet}(\mathcal{T}_{\overline{G/N}}))}
\newcommand{\ringOfFunctionsForBasicAffineSpace}{A}
\newcommand{\ringOfFunctionsForCOTANGENTBUNDLEOfBasicAffineSpace}{R}
\newcommand{\tangentSheafForBasicAffineSpace}{\mathcal{T}_{G/N}}
\newcommand{\symOfDirectSumOfRepsofFundamentalWeights}{\text{Sym}(\oplus_i E(\omega_i))}

\newcommand{\projectionFromAffineClosureofCotangentBundleToAffineClosureofSpace}{\overline{\pi}}
\newcommand{\momentMapFromAFFINECLOSUREofCotangentSpaceWithGROUPG}{\overline{\mu}_G}
\newcommand{\momentMapFromAFFINECLOSUREofCotangentSpacewithTOTALGROUP}{\overline{\mu}}
\newcommand{\groundfield}{k}
\newcommand{\LGd}{\mathfrak{g}^*}
\newcommand{\momentMapFromAFFINECLOSUREofCotangentSpaceWithGROUPT}{\overline{\mu}_T}

\newcommand{\LeviSubgroup}{L}
\newcommand{\borelOfLevi}{B_L}
\newcommand{\unipotentRadicalofLevi}{N_L}
\newcommand{\lieAlgebraofUnipotentRadicalofLevi}{\mathfrak{n}_L}
\newcommand{\universalPartialHyperkahlerImplosionwithParabolic}{G \mathop{\times}\limits^{\unipotentRadicalOfPARABOLICSUBGROUP} (\mathfrak{g}/\lieAlgebraOfUnipotentRadicalOfPARABOLICSUBGROUP)^*}
\newcommand{\rhocheck}{\rho^{\vee}}
\newcommand{\unipotentRadicalOfPARABOLICSUBGROUP}{U_P}
\newcommand{\lieAlgebraOfUnipotentRadicalOfPARABOLICSUBGROUP}{\mathfrak{u}_P}
\newcommand{\affinizationOfGrothendieckSpringerResolution}{\LGd \times_{\LTd\sslash W} \LTd}
\newcommand{\smoothLocusOfAffineclosureBasicAffineSpace}{\overline{G/N}^{\text{sm}}}

\newcommand{\ringOfFunctionsForFIBERPRODUCTofCotangentBundleofGModNOverLieTWithZero}{\overline{\ringOfFunctionsForCOTANGENTBUNDLEOfBasicAffineSpace}}
\newcommand{\affineClosureFIBERPRODUCTOfCotangentBundleofBasicAffineSpaceatZeroOverLT}{\affineClosureOfCotangentBundleofBasicAffineSpace \times_{\LTd} \{0\}}
\newcommand{\WhittakerHamiltonianReductionofG}{T^{\ast}(G/_{\psi}N^-)}
\newcommand{\WhittakerHamiltonianReductionofGBASECHANGEDtoLTd}{T^{\psi}_{\LTd}}
\newcommand{\RingOfFunctionsForWhittakerHamiltonianReductionOfGBASECHANGEDtoLTd}{\mathcal{W}}
\newcommand{\affineGrassmannianForLANGLANDSDUALgroup}{\text{Gr}_{G^{\vee}}}
\newcommand{\regularRepresentationPerverseSheafOnAffineGrassmannianForLANGLANDSDUALgroup}{\mathcal{R}_G}
\newcommand{\equivariantNilHeckeRingforW}{\mathscr{N}}

\newcommand{\LanglandsdualtoG}{G^{\vee}}
\newcommand{\LanglandsDualParabolic}{P^{\vee}}
\newcommand{\LanglandsDualTorus}{T^{\vee}}
\newcommand{\LanglandsDualBorel}{B^{\vee}}
\newcommand{\SymtWITHOPPOSITEGRADING}{\overline{\Symt}}
\newcommand{\LieAlgebraofCommutatorofOPPOSITEBorel}{\overline{\mathfrak{n}}}
\maketitle

\newcommand{\quasiMinusculeCOWEIGHTofG}{\gamma}
\newcommand{\intersectionCohomologySheafForQuasiMinusculeCOWEIGHTofG}{\text{IC}_{\quasiMinusculeCOWEIGHTofG}}
\newcommand{\Symtd}{\text{Sym}(\LTd)}
\newcommand{\SymtdAdjoinHbar}{\text{Sym}(\LTd)[\hbar]}
\newcommand{\equivariantCohomologyRingofPointforTorusTIMESGM}{R}
\newcommand{\inclusionOfAllFixedPointsIntoGModP}{q}
\newcommand{\chernClassOfLineBundleOnGrThetaVee}{c_1(\theta)}
\newcommand{\compactifiedLineBundleInGrGforQuasiMinusculeWeight}{\mathcal{L}_{\theta}}
\newcommand{\functionsToAdjoinToGetEquivariantHomologyOfNgoPoloLineBundle}{\mathcal{F}}
\newcommand{\functionsOnLTTILDEtimesLTDTILDEmodWPOverCenter}{A}

\newcommand{\generalMomentGraph}{\mathscr{G}}
\newcommand{\blowUpofTHISbyTHISandTHISwithStrictTransformOfLASTTHISRemoved}[3]{\mathcal{B}(#1, #2, #3)}
\newcommand{\universalCentralizersforTHISGROUPOverLTd}[1]{\mathfrak{Z}_{#1, \LTd}}
\newcommand{\blowUpOfTHISClosedSubschemeOfTHAT}[2]{\text{Bl}_{#1}(#2)}
\newcommand{\affineBlowUpByThisSchemeInThisSchemeAndThisFunction}[3]{\text{ABl}_{#1}(#2, #3)}
\newcommand{\universalCentralizersforTHISGROUPOverLTdMODTHISBASE}[2]{\mathfrak{Z}_{#1}^{\mathfrak{t}^*\sslash W_{#2}}}
\newcommand{\grothendieckSpringerResolutionForTHISPARABOLIC}[1][]{\tilde{\mathfrak{g}}_{#1}}
\newcommand{\grothendieckSpringerResolutionREGULARPARTFORBOTHForTHISPARABOLIC}[1][]{\tilde{\mathfrak{g}}_{#1\text{reg}}^{\text{reg}}}
\newcommand{\LP}{\mathfrak{p}}

\newcommand{\baseOfValpha}{Y_{\alpha}}
\newcommand{\fixedPointsForTildeTActionOnCompactifiedLineBundle}{\compactifiedLineBundleInGrGforQuasiMinusculeWeight^{\tilde{T}}}
\newcommand{\tildeTEquivariantHOMOLOGYofCompactifiedLineBundle}{H^{\tilde{T}}_{BM}(\compactifiedLineBundleInGrGforQuasiMinusculeWeight)}
\newcommand{\setOfPairsGivingFixedPointsOfCompactifiedLineBundle}{\mathcal{P}_{\theta}}
\newcommand{\tildeTEquivariatnCohomologyofCompactifiedLineBundle}{H_{\tilde{T}}^*(\compactifiedLineBundleInGrGforQuasiMinusculeWeight)}
\newcommand*{\HarishChandraBimodulesForTHISGROUPDefaultsToG}[1][G]{\text{HC}_{#1}}

\newcommand{\WhitofLGMod}{\text{Whit}(\LG\text{-mod})}
\newcommand{\Ug}{U\LG}
\newcommand{\Uhbarg}{U_{\hbar}\LG}
\newcommand{\Uhbarn}{U_{\hbar}\LN}
\newcommand{\Uhbarnminus}{U_{\hbar}\LN^-}
\newcommand{\filteredVectorSpacesAskhbarModCompatibly}{k[\hbar]\text{-mod}_{\text{gr}}}
\newcommand{\AsymptoticHarishChandraBimodulesForTHISGROUPDefaultsToG}[1][G]{\text{HC}_{G, \hbar}}

\newcommand{\Zhbarg}{Z_{\hbar}\LG}

\maketitle
\begin{abstract}
We prove that the algebra of functions on the cotangent bundle $T^*(G/U_P)$ of the parabolic base affine space for a reductive group $G$ and a parabolic subgroup $P$ is isomorphic to the subalgebra of the functions on $G \times L \times \mathfrak{l}\sslash L$ which are invariant under a certain action of the group scheme of universal centralizers on $G$, where $L$ is a Levi subgroup of $P$ and $\mathfrak{l}$ is its Lie algebra, upgrading an isomorphism of Ginzburg and Kazhdan simultaneously to the parabolic and the modular setting. We also derive a related isomorphism for the partial Whittaker cotangent bundle of G, which proves a conjecture of Devalapurkar.
\end{abstract}
\newcommand{\LHd}{\mathfrak{h}^*}
\renewcommand{\F}{\mathcal{F}}
\newcommand{\overlineU}{\overline{U}}
\newcommand{\LUbar}{\overline{\mathfrak{u}}}
\section{Introduction}
The main result of this article, \cref{thm: intro main}, generalizes an \lq implosion\rq{} description of \cite{GinzburgKazhdanDifferentialOperatorsOnBasicAffineSpaceandtheGelfandGraevAction} for the functions on the cotangent bundle of the basic affine space of a complex reductive group simultaneously to the parabolic setting and to the modular setting. Using \cref{thm: intro main}, we also derive an implosion description for the functions on the \textit{partial Whittaker cotangent bundle} of a reductive group, stated precisely in \cref{Isomorphism Induced by L Whittaker Reduction Corollary}, which proves an isomorphism whose existence was conjectured by Devalapurkar \cite[Conjecture 3.6.15]{Devalapurkar-ku-relative-Langlands}. 

Before stating our main theorem, we set some notation. Let $G$ denote some reductive group defined over some algebraically closed field $k$ whose characteristic does not lie in some \lq small\rq{} set of characteristics described explicitly in Section \labelcref{Notation Subsection} (for example, if $G = \GL_n$ we impose no restriction on the characteristic of $k$) and let $P$ denote some parabolic subgroup of $G$. Choose some Levi factor $L$ in $P$, so that we have a semidirect product decomposition $U_P \rtimes L \cong P$ where $U_P$ is the unipotent radical of $P$. Let $\LG, \LP, \mathfrak{l}$ and $\mathfrak{u}_P$ denote the Lie algebras of their respective groups. We set $\mathfrak{c}_L := \Spec(\Sym(\mathfrak{l})^L)$, and denote by $J_G$ the \textit{group scheme of universal centralizers} studied in \cite[Section 2]{NgoLeLemmeFondamentalPourLesAlgebresdeLie}, which we describe more precisely in Section \labelcref{Notation Subsection}. The main result of this article is the following: 
\begin{Theorem}\label{thm: intro main}
    There is an action of $J_G$ on $G \times \mathfrak{c}_L \times L$ inducing an isomorphism of algebras \begin{equation}\label{Isomorphism Without Whittaker}\mathcal{O}(T^*(G/U_P)) \cong \mathcal{O}(G \times \mathfrak{c}_L \times L)^{J_G}\end{equation} compatible with the natural actions of $G$ and $L$.
\end{Theorem}

In fact, we show slightly more: we show that the right hand side of \cref{Isomorphism Without Whittaker} naturally acquires a $\Sym(\LG \oplus \mathfrak{l})$-algebra structure and show this isomorphism is an isomorphism of $\Sym(\LG \oplus \mathfrak{l})$-algebras.  

The ring $\mathcal{O}(T^*(G/U_P))$ has been previously studied due to its appearance in the geometric Langlands program \cite{MaceratoLEviEquivariantRestrictionofSphericalPerverseSheaves}, \cite{Devalapurkar-ku-relative-Langlands} as well in the study of 3d $\mathcal{N} = 4$ supersymmetric gauge theories and their Coulomb branches \cite{BourgetDancerGrimmingerHananyZhongPartialImplosionsandQuivers}, \cite{GannonWilliamsDifferentialOperatorsOnBaseAffineSpaceofSLnandQuantizedCoulombBranches}, \cite{DancerGrimmingerMartensZhongComplexSymplectic}. Our result may be useful in studying in the conjectures made in \cite{BourgetDancerGrimmingerHananyZhongPartialImplosionsandQuivers} and \cite{DancerGrimmingerMartensZhongComplexSymplectic}. Moreover, as we explain further in \cref{BZSV Connection}, one can use \cref{thm: intro main} to derive a natural candidate for the dual of the Hamiltonian $G$-variety $T^*(G_{\mathbb{C}}/L_{\mathbb{C}})$ in the relative Langlands duality program \cite{BenZviSakleredisVenkateshRelativeLanglandsDuality}, where $L_{\mathbb{C}}$ is a Levi subgroup of $G_{\mathbb{C}}$.

\subsection{Acknowledgments}
This project grew out of discussions with Sanath Devalapurkar, who decided not to sign it in the capacity of author. I would like to especially thank him for numerous useful discussions related to this paper; in particular, his suggestion to use Kostant-Whittaker descent (see Section \labelcref{Kostant Whittaker Descent Subsection}) greatly simplified the original argument the author had in mind to prove \cref{thm: intro main}. I would also like to thank Victor Ginzburg, Mark Macerato, Kendric Schefers, and Aaron Slipper for interesting and useful comments.

\section{Proof of the Main Theorem}
\newcommand{\bigCell}{\mathscr{B}}
\subsection{Notation}\label{Notation Subsection}  We assume $G$ is a pinned reductive group over an algebraically closed field of characteristic $p \geq 0$ which satisfies \cite[Condition (C4)]{RicheKostantSectionUniversalCentralizerandModularDerivedSatakeEquivalence}. In particular, we have chosen a  maximal torus $T$, a choice of Borel subgroup $\overline{B}$ containing $T$, and a simple root vector $f_i$ in each simple root space of $\LUbar := \text{Lie}(\overlineU)$, where $\overline{U} := [\overline{B}, \overline{B}]$. Let $B$ be the opposite Borel to $\overline{B}$ which contains $T$. Let $X_{\bullet}(T)$ denote the lattice of cocharacters and $Z\Phi^{\vee}$ denote the coroot lattice. By definition, $p$ satisfies condition (C4) if and only if $p$ is good for $G$ in the sense of \cite[Definition 4.22]{JantzenRepresentationsofAlgebraicGroups} (which is automatically satisfied if $p > 5$ or $p = 0$), the quotient $X_{\bullet}(T)/Z\Phi^{\vee}$ has no $p$-torsion (which in \cite[Section 2.2]{RicheKostantSectionUniversalCentralizerandModularDerivedSatakeEquivalence} is observed to hold if the derived subgroup of $G$ is simply connected) and there exists a $G$-equivariant isomorphism $\LG \xrightarrow{\sim} \LGd$. Let $\mu_{\LGd}: \LGd \to \LUbar^*$ denote the map induced by pullback of the inclusion $\LUbar \subseteq \LG$.

\subsubsection{Parabolic Subgroup Notation} Fix a subset $I \subseteq \Delta$ of the simple roots $\Delta$. Let $P \supseteq B$ denote the standard parabolic subgroup determined by $I$, or, in other words, the smallest closed subgroup scheme containing $B$ whose Lie algebra contains the $f_i$. 
Let $U_P$ denote the unipotent radical of $P$ and $L$ denote the quotient $P/U_P$. Our pinning gives a Levi decomposition $P \xleftarrow{\sim} U_P \rtimes L$, so that we may identify $L$ as a subgroup of $G$.  This semidirect product decomposition in turn indues a semidirect product decomposition $U \cong U_P \rtimes U_L$ where $U$ is the unipotent radical of $B$. In particular, we have a group isomorphism $U/U_P \cong U_L$ where $U_L$ is the unipotent radical of $L$. Letting $\overline{P}$ denote the unique parabolic subgroup conjugate to $P$ and containing $\overline{B}$ and $\overline{U}^P$ denote its unipotent radical, we have a similar semidirect product decomposition $\overline{U} \xleftarrow{\sim} \overline{U}^P \rtimes \overline{U}^L$. 

Our pinning gives a canonical identification $\overline{U}/[\overline{U}, \overline{U}] \cong \prod_{\Delta} \mathbb{G}_a$, which follows using the isomorphism of \cite[II.1.7(1)]{JantzenRepresentationsofAlgebraicGroups}, say, 
and so using this identification we may define the character $\psi$ as the composite \[\overline{U} \xrightarrow{} \overline{U}/[\overline{U}, \overline{U}] \cong \prod_{\Delta} \mathbb{G}_a \xrightarrow{\Sigma} \mathbb{G}_a\] where $\Sigma$ is the sum map. We use the notation $d\psi_e: \LUbar \to \mathbb{A}^1$ for the map induced by $\psi$ on the underlying Lie algebras. 

We define $\psi_P := \psi|_{\overline{U}^P}$ and let $\psi_L: \overline{U}/\overline{U}^P \to \mathbb{G}_a$ denote the character $\psi - \psi_P$, which we regard as a character on $\overline{U}^{L}$. 

\subsubsection{Kostant Section} We choose, once and for all, a Kostant section \[\kappa: \mathfrak{c} := \Spec(\Sym(\LG)^G) \to \mu_{\LGd}^{-1}(d\psi_e) := \mathfrak{g}^* \times_{\LUbar^*} \{d\psi_e\}\] which splits the quotient map $q_G: \LGd \to \mathfrak{c}$, which we may do by for example using our $G$-equivariant isomorphism $\LG \xrightarrow{\sim} \LGd$ and the results of \cite[Section 3]{RicheKostantSectionUniversalCentralizerandModularDerivedSatakeEquivalence}. We let $J_G$ denote the centralizer of this Kostant section: in other words, we define $J_G$ as the closed subscheme of $G \times \mathfrak{g}^*$ for which the diagram
$$\xymatrix{
    J_G \ar[r] \ar[d]^{\subseteq} & \mathfrak{c} \ar[d]^{\xi \mapsto (\kappa(\xi), \kappa(\xi))} \\
   G \times \LGd \ar[r]^{\mathrm{act}, \mathrm{proj}} & \g^\ast \times \g^{\ast}.
    }$$ is Cartesian. Observe that $J_G$ naturally acquires the structure of a group scheme over $\mathfrak{c}$.

As above, we let $\mathfrak{c}_L := \Spec(\Sym(\mathfrak{l})^L)$, which we always view as a scheme over $\mathfrak{c}$ by the Chevalley restriction map. Our choice of Kostant section identifies $J_G$ with the \textit{group scheme of universal centralizers} studied in \cite[Section 2]{NgoLeLemmeFondamentalPourLesAlgebresdeLie}, as explained in for example \cite[Section 3.3]{RicheKostantSectionUniversalCentralizerandModularDerivedSatakeEquivalence}. 

\subsubsection{Notation for Whittaker Reduction}Let $N$ denote some arbitrary unipotent group. If $X$ is some variety equipped with a free $N$-action and we are given an $N$-equivariant map $\mu: X \to \mathfrak{n}^*$, for any (additive) character $\alpha: N \to \mathbb{G}_a$ we set \[X\sslash_{\alpha} N := (X \times_{\mathfrak{n}^{*}} \{d\alpha_e\})/N.\] where $d\alpha_e$ is the induced character on the respective Lie algebras. If $X = T^*(Y)$ for some $Y$ with an $N$-action, we also set $T^*(Y/_{\alpha}N) := T^*(Y)\sslash_{\alpha} N$.

\subsection{Whittaker Reduction}\label{Whit Subsection}In this section, we state a small upgrade of our main result, \cref{Restated Theorem Intro Main for Whittaker}, and argue that implies \cref{thm: intro main}:

\begin{Theorem}\label{Restated Theorem Intro Main for Whittaker}
    There is an action of $J_G$ on $T^*(\overline{U}^L_{\psi_L}\backslash L)$ for which the induced diagonal $J_G$-action induces an isomorphism of $\Sym(\LG \oplus \mathfrak{l})$-algebras
    $$\mathcal{O}(T^\ast(G/U_P)) \cong \mathcal{O}(T^*(G/_{\psi}\overline{U}) \times_{\mathfrak{c}} T^*(\overline{U}^L_{\psi_L}\backslash L))^{J_G}$$ compatible with the actions of $G$ and $L$.
\end{Theorem}

To prove \cref{Restated Theorem Intro Main for Whittaker} implies \cref{thm: intro main}, we first prove the following elementary lemma, see also \cite[Lemma 3.2.3]{GinzburgKazhdanDifferentialOperatorsOnBasicAffineSpaceandtheGelfandGraevAction}:

\begin{Lemma}\label{Scheme with Almost Hamiltonian G Action Has Easy Whittaker Reduction After Choosing Kostant Section}
    If $X$ is a scheme with a $G$-action and a $G$-equivariant map $X \to \LGd$, then there is an isomorphism $X \sslash_{\psi} \overline{U} \cong X \times_{\LGd} \mathfrak{c}$.
\end{Lemma}

\begin{proof}
    By definition, the scheme $X \sslash_{\psi} \overline{U}$ is obtained by the quotient of the scheme \[X \times_{\overline{\mathfrak{u}}^*} \{d\psi_e\} \cong X \times_{\LGd} (\LGd \times_{\overline{\mathfrak{u}}^*} \{d\psi_e\})
    \] by the diagonal $\overline{U}$-action. However, the product $(\LGd \times_{\overline{\mathfrak{u}}} \{d\psi_e\})$ is a $\overline{U}$-torsor which is trivialized by $\kappa$ \cite{KostantonWhittakerVectorsandRepresentationTheory}, \cite[Proposition 3.2.1]{RicheKostantSectionUniversalCentralizerandModularDerivedSatakeEquivalence}. 
    Therefore the quotient of this scheme is canonically isomorphic to $X \times_{\LGd} \mathfrak{c}$, as required. 
\end{proof}

Applying \cref{Scheme with Almost Hamiltonian G Action Has Easy Whittaker Reduction After Choosing Kostant Section} for both $G$ and $L$, we obtain an isomorphism \[\mathcal{O}(T^*(G/_{\psi}\overline{U}) \times_{\mathfrak{c}} T^*(\overline{U}^L_{\psi_L}\backslash L))^{J_G} \cong \mathcal{O}((G \times \mathfrak{c}) \times_{\mathfrak{c}} (L \times \mathfrak{c}_L))^{J_G}\] compatible with the actions of $G$ and $L$. Thus, identifying \[(G \times \mathfrak{c}) \times_{\mathfrak{c}} (L \times \mathfrak{c}_L) \cong G \times L \times \mathfrak{c}_L\] we transport the induced $J_G$-action and see that \cref{thm: intro main}, as well as its variant stated in the abstract, are implied by \cref{Restated Theorem Intro Main for Whittaker}.

\subsection{Kostant-Whittaker Descent}\label{Kostant Whittaker Descent Subsection}Observe that the natural action of $G \times G$ on $G \times \LGd$ induces an action of $G \times J_G$ on $G \times \mathfrak{c}$. In particular, to any $M \in \text{Rep}(J_G)$ we may define the quasicoherent sheaf \[(a_*(\mathcal{O}_{G \times \mathfrak{c}}) \otimes_{\mathcal{O}(\mathfrak{c})} M)^{J_G}\] on $\LGd_{\text{reg}}$, where $a: G \times \mathfrak{c} \to \mathfrak{g}_{\text{reg}}^*$ is the action map. Since $\mathcal{O}_{G \times \mathfrak{c}}$ is evidently flat as an $\mathcal{O}(\mathfrak{c})$-module, this quasicoherent sheaf canonically acquires a $G$-equivariant structure. In this section, we construct the following isomorphism:

\begin{Proposition}\label{Any algebra object on regular locus is recovered from its Kostant-Whittaker Reduction}
    To any $\F \in \QCoh(\LG_{\text{reg}})^G$, there is a canonical isomorphism \[\F \xrightarrow{\sim} (a_*(\mathcal{O}_{G \times \mathfrak{c}}) \otimes_{\mathcal{O}(\mathfrak{c})} \kappa^*(\F))^{J_G}\] which is an isomorphism of ring objects if $\F$ is a ring object in $\QCoh(\LG_{\text{reg}})^G$.
\end{Proposition}

We do this after first proving the following lemma:

\begin{Lemma}\label{Kostant-Whittaker Reduction Is Exact Monoidal Equivalence of Categories with Explicitly Constructed Inverse}
    The functor $\kappa^*$ lifts to an exact monoidal equivalence of categories \[\kappa^*: \QCoh(\LG_{\text{reg}})^G \xrightarrow{\sim} \mathcal{O}(J_G)\text{-comod}\] whose inverse is given by the functor $M \mapsto (a_*(\mathcal{O}_{G \times \mathfrak{c}}) \otimes_{\mathcal{O}(\mathfrak{c})} M)^{J_G}$.
\end{Lemma}

\begin{proof}
    The fact that such an equivalence of categories exists is standard and is proved, for example, in \cite[Proposition 3.3.11]{RicheKostantSectionUniversalCentralizerandModularDerivedSatakeEquivalence}; we now explicitly compute its inverse. The action map $a$ is smooth and surjective \cite[Lemma 3.3.1]{RicheKostantSectionUniversalCentralizerandModularDerivedSatakeEquivalence}. Thus, using the computation of \cite[Proposition 3.3.11]{RicheKostantSectionUniversalCentralizerandModularDerivedSatakeEquivalence}, we see that descent theory gives adjoint equivalences of abelian categories \begin{equation}\label{Equivalence From Descent}a^*: \QCoh(\LGd_{\text{reg}})^{G} \leftrightarrow{} \QCoh(G \times \mathfrak{c})^{G \times J_G}: a_*(-)^{J_G}\end{equation} which thus are in particular exact. It is standard (and not difficult to check) that we have equivalences of abelian categories \begin{equation}\label{Equivalence from Stalk}(e, \text{id})^*: \QCoh(G \times \mathfrak{c})^{G \times J_G} \leftrightarrow{} \QCoh(\mathfrak{c})^{J_G}: p^*\end{equation} where $e: \Spec(k) \to G$ is the identity point and we identify $(e, \text{id})^*$ with $(-)^G$. 
    Therefore these functors are in particular adjoint. Combining these two adjunctions and using the canonical identification of $(e \times \text{id})^*a^*$ with $\kappa^*$ we obtain our desired inverse, as any adjoint to an equivalence of categories gives an inverse.
\end{proof}

\begin{proof}[Proof of \cref{Any algebra object on regular locus is recovered from its Kostant-Whittaker Reduction}]
    Our desired isomorphism is given by the unit of the monoidal adjunction of \cref{Kostant-Whittaker Reduction Is Exact Monoidal Equivalence of Categories with Explicitly Constructed Inverse}. Since $\kappa^*$ is a monoidal equivalence of categories, both $\kappa^*$ and its right adjoint are monoidal, and so the unit map induces an isomorphism of ring objects. Explicitly, this isomorphism is given by the composite \[(\mathcal{O}_{G \times \mathfrak{c}} \otimes_{\mathcal{O}(\mathfrak{c})} \kappa^*(\F))^{J_G} = (\mathcal{O}_{G \times \mathfrak{c}} \otimes_{\mathcal{O}(\mathfrak{c})} (\mathcal{O}(\mathfrak{c}) \otimes_{\mathcal{O}_{\LGd_{\text{reg}}}} \F))^{J_G} \xleftarrow{\sim} (\mathcal{O}_{G \times \mathfrak{c}} \otimes_{\mathcal{O}_{\LGd_{\text{reg}}}} \F)^{J_G} \xleftarrow{\sim} \F\] 
    where the first equivalence is given by the unit of \labelcref{Equivalence from Stalk} and the second equivalence is given is the tensor-hom adjunction given by \labelcref{Equivalence From Descent}.
\end{proof}

\subsection{Restriction to the Big Cell}
The inclusion \[\bigCell := \overline{U}^PP/U_P \xhookrightarrow{} G/U_P\] is an open embedding of a subscheme invariant under the $\overline{U}$-action. Thus we obtain an open embedding of the cotangent bundles $T^*(\overline{U}^PP/U_P) \xhookrightarrow{} T^*(G/U_P)$ which respects the induced Hamiltonian $\overline{U}$-action, and so we obtain an open embedding $j: T^*(\overline{U}{}_{\psi}\backslash\overline{U}^PP/U_P) \xhookrightarrow{} T^*(\overline{U}{}_{\psi}\backslash G/U_P)$. 

\begin{Proposition}\label{Open Embedding of Big Cell Induces Isomorphism on Whittaker Cotangent Bundles}
    The map $j$ is an isomorphism.
\end{Proposition}

We prove this after setting some notation which will also be used later.  Since multiplication induces a $\overline{U}^P$-equivariant isomorphism \[\overline{U}^P \times P/U_P \xrightarrow{\sim} \overline{U}^PP/U_P\]
we have an isomorphism \[T^*(\overline{U}^P_{\psi_P}\backslash\bigCell) \cong T^*(\overline{U}^P_{\psi_P}\backslash \overline{U}^P \times P/U_P) \xrightarrow{\sim} T^*(P/U_P)\] of Hamiltonian $U_L$-varieties. It is standard that $T^*(\overline{U}{}_{\psi}\backslash\bigCell)$ identifies with the Kostant-Whittaker reduction of the left Hamiltonian $L$-space $T^*(\overline{U}^P_{\psi_P}\backslash\bigCell)$ by $\psi_L$, and so we obtain an isomorphism \[h: T^*(\overline{U}{}_{\psi}\backslash\bigCell) \xrightarrow{\sim} T^*(\overline{U}^L_{\psi_L}\backslash L)\] of right Hamiltonian $L$-varieties. 
\begin{proof}[Proof of \cref{Open Embedding of Big Cell Induces Isomorphism on Whittaker Cotangent Bundles}:]Let $\mu_G$, respectively $\mu_L$, denote the moment map for the left Hamiltonian $G$-action, respectively right Hamiltonian $L$-action, on $T^*(G/U_P)$, and let $\mu_L^{\circ}$ denote the moment map for the right Hamiltonian $L$-action on $T^*(\overline{U}{}_{\psi}\backslash  G/U_P)$. Since $\overline{U}^PP/U_P \subseteq G/U_P$ is closed under the right $L$-action, the moment map for the right Hamiltonian $L$-action on $T^*(\overline{U}{}_{\psi}\backslash\overline{U}^PP/U_P)$ is $\mu_L^{\circ} \circ j$. Let \[t: T^*(\overline{U}{}_{\psi}\backslash G/U_P) \to \mathfrak{c}_L := \Spec(\Sym(\mathfrak{l})^L)\] denote the composite of $\mu_L^{\circ}$ with the quotient map $q_L: \mathfrak{l}^* \to \mathfrak{c}_L$. Since any map of torsors is an isomorphism, it suffices to show that $t$ and $tj$ are $L$-torsors.

    Using the map $h$ and \cref{Scheme with Almost Hamiltonian G Action Has Easy Whittaker Reduction After Choosing Kostant Section}, we see that the map $tj$ is a trivial $L$-torsor. Thus it remains to show that $t$ is an $L$-torsor. To this end, first observe that, if we let $\beta$ denote the restriction of $\mu_G$ to the Kostant section, there is a Cartesian square\begin{equation}\label{Observed Cartesian Diagram}\xymatrix{
    T^\ast(\overline{U}{}_\psi \backslash G/U_P) \ar[r]^{\subseteq} \ar[d]^{\beta} & T^\ast(G/U_P) \ar[d]^{\mu_G} \\
   \fr{c} \ar[r]^{\kappa} & \g^\ast.
    }\end{equation}
    by \cref{Scheme with Almost Hamiltonian G Action Has Easy Whittaker Reduction After Choosing Kostant Section}. Since the image of $\kappa$ factors through the regular locus $\LGd_{\mathrm{reg}}$ of $\LGd$, we obtain a commutative diagram 
    \begin{equation}\label{Big Diagram}\xymatrix{
    T^\ast(\overline{U}{}_\psi \backslash G/U_P) \ar[r]^{\subseteq} \ar[d]^{\tilde{\beta}}  & T^\ast(G/U_P) \times_{\LGd} \LGd_{\mathrm{reg}} \ar[d]^{\tilde{\mu}_{reg}} \\
    \mathfrak{c} \times_{\LGd_{\mathrm{reg}}} \tilde{\g}_P^{\mathrm{reg}}  \ar[r]^{\kappa'} \ar[d] & \tilde{\g}_P^{\mathrm{reg}} \ar[d] \ar[r]^{\overline{\mu}_L|_{\tilde{\g}_P^{\mathrm{reg}} }} & \mathfrak{c}_L \ar[d] \\
    \fr{c} \ar[r]^{\kappa} & \g^\ast_{\mathrm{reg}} \ar[r]^{q_G|_{\LGd_{\mathrm{reg}}}} & \fr{c},
    }\end{equation} 
    where $\tilde{\g}_P$ is the (dual) parabolic Grothendieck-Springer resolution $G \times^{P} (\LG/\mathfrak{u}_P)^*$, $\tilde{\g}_P^{\mathrm{reg}}$ is its restriction to the regular locus of $\LGd$, $\tilde{\mu}$ is the composite \begin{equation}\label{Definition of TildeMu}T^*(G/U_P) \cong G \times^{U_P} (\LG/\mathfrak{u}_P)^* \to G \times^{P} (\LG/\mathfrak{u}_P)^* = \tilde{\g}_P,\end{equation} $\tilde{\mu}_{\mathrm{reg}}$ is its restriction to the regular elements of $\LGd$, $\overline{\mu}_L$ is the map induced by fact that $\mu_L$ is equivariant for the right $L$-action,  $\kappa'$ is projection, and $\tilde{\beta}$ is the map induced by $\beta$ and $\tilde{\mu}_{\mathrm{reg}}$. 
    
    We claim that every square in \labelcref{Big Diagram} is Cartesian: indeed, the fact that the bottom left square is Cartesian follows by definition, and this, along with the fact \labelcref{Observed Cartesian Diagram} is Cartesian, implies the top square is Cartesian. Finally, the fact the rightmost square is Cartesian is known, see for example \cite{KostantLieGroupRepsonPolynomialRings}, \cite[Remark 3.5.4]{RicheKostantSectionUniversalCentralizerandModularDerivedSatakeEquivalence}. 
    
    Now, as $\kappa$ is a section of $q_G$, by definition, the composite $q_G \kappa$ is the identity, and so the composite $\overline{\mu}_L\kappa'$ is an isomorphism since the bottom two squares in \labelcref{Big Diagram} are Cartesian. Because of this and the fact that \[t := q_L\mu_L^{\circ} = q_L\mu_L|_{T^*(\overline{U}{}_\psi \backslash G/U_P)} = \overline{\mu}_L\tilde{\mu}_{\mathrm{reg}}|_{T^*(\overline{U}{}_\psi \backslash G/U_P)} = \overline{\mu}_L\kappa'\tilde{\beta},\] to show that $t$ is an $L$-torsor it suffices to show that $\tilde{\beta}$ is an $L$-torsor. However, this follows directly from the fact that $\tilde{\mu}$ is evidently an $L$-torsor by its definition in \labelcref{Definition of TildeMu}, as well as the fact that the top square of the diagram \labelcref{Big Diagram} is Cartesian.\footnote{This proof was developed in discussions with Sanath Devalapurkar.}
\end{proof}
\subsection{Recovering Functions from the Regular Locus}\label{Codimension Subsection}
The goal of this section will be to prove the following proposition, which allows us to recover the functions on $T^*(G/U_P)$ from its restriction \[T^*(G/U_P)_{\text{reg}} := T^*(G/U_P) \times_{\LGd} \LGd_{\text{reg}}\] to the regular locus, and whose proof will occupy the entirety of Section \labelcref{Codimension Subsection}.

\begin{Proposition}\label{Global Functions on Cotangent Bundle is Global Functions on Cotangent Bundle of G Mod UP}
    The restriction map induces an isomorphism $\mathcal{O}(T^*(G/U_P)) \xrightarrow{\sim} \mathcal{O}(T^*(G/U_P)_{\text{reg}})$.
\end{Proposition}

The cotangent bundle $T^*(G/U_P)$ is well known to be identified with the variety $(G \times (\mathfrak{g}/\mathfrak{u}_P)^*)/U_P$ and, under this identification, the open subscheme $T^*(G/U_P)_{\text{reg}}$ corresponds to the variety $(G \times (\mathfrak{g}/\mathfrak{u}_P)_{\text{reg}}^*)/U_P$. Thus since $T^*(G/U_P)$ is a smooth (and, in particular, normal) variety, \cref{Global Functions on Cotangent Bundle is Global Functions on Cotangent Bundle of G Mod UP} follows immediately from the following lemma:

\begin{Lemma}\label{Complement of Regular Locus of Parabolic Has Codimension Two}
    The codimension of the complement of $(\mathfrak{g}/\mathfrak{u}_P)^*_{\text{reg}}$ in $(\mathfrak{g}/\mathfrak{u}_P)^*$ is at least two. 
\end{Lemma}

\begin{proof}
Let $\LG^{*, c}_{\mathrm{reg}}$ denote the complement of the regular elements in $\LGd$, so that the set \[(\mathfrak{g}/\mathfrak{u}_P)_{\mathrm{reg}}^{*, c} := \LG^{*, c}_{\mathrm{reg}} \cap (\mathfrak{g}/\mathfrak{u}_P)^*\] is the complement of $(\mathfrak{g}/\mathfrak{u}_P)^*_{\text{reg}}$ in $(\mathfrak{g}/\mathfrak{u}_P)^*$. Assume that $V$ is a nonempty irreducible component of $(\mathfrak{g}/\mathfrak{u}_P)^{*, c}_{\mathrm{reg}}$ that has codimension one in $(\mathfrak{g}/\mathfrak{u}_P)^*$. Since $\mathcal{O}((\LG/\mathfrak{u}_P)^*)$ is a polynomial algebra, we may write $V$ as the vanishing set $V(f)$ of some function $f \in \mathcal{O}((\LG/\mathfrak{u}_P)^*)$ and, since $V$ is nonempty and invariant under scalar multiplication, we see that $f$ is a homogeneous polynomial of some positive degree. 

Let $\LG^{*}_{\mathrm{rss}}$ denote the set of regular semisimple elements in $\LGd$, and let $\LG^{*, c}_{\mathrm{rss}}$ denote its complement. Since $\LGd_{\mathrm{rss}} \subseteq \LGd_{\mathrm{reg}}$, $V$ is an irreducible closed subscheme of \[(\mathfrak{g}/\mathfrak{u}_P)_{\mathrm{rss}}^{*, c} := \LG^{*, c}_{\mathrm{rss}} \cap (\LG/\mathfrak{u}_P)^*\] or, in other words, of the complement of the set of regular semisimple elements in $(\mathfrak{g}/\mathfrak{u}_P)^*$. Letting $\mathfrak{u}$ denote the Lie algebra of the unipotent radical of $B$, we may restrict to the closed subscheme $(\LG/\mathfrak{u})^* \subseteq (\LG/\mathfrak{u}_P)^*$ and obtain that \begin{equation}\label{Easier Containment}V(f|_{(\LG/\mathfrak{u})^*}) = V(f) \cap (\LG/\mathfrak{u})^* \subseteq \LG^{*, c}_{\mathrm{rss}} \cap (\LG/\mathfrak{u})^*.\end{equation} Furthermore, the proof of \cite[Proposition 1.9.3]{BezrukavnikovRicheAffineBraidGroupActionsonDerivedCategoriesofSpringerResolutions} 
shows that \begin{equation}\label{Harder Containment} \LG^{*, c}_{\mathrm{rss}} \cap (\LG/\mathfrak{u})^* = \cup_{\alpha}V(h_{\alpha})\end{equation} for certain homogeneous polynomials $h_{\alpha} \in \mathcal{O}((\LG/\mathfrak{u})^*)$ of positive degree for which $V(h_{\alpha}) \cap (\LG/\mathfrak{u})^*_{\text{reg}}$ is nonempty for our choice of $p$, and more generally any good prime $p$. 
Combining \labelcref{Easier Containment} and \labelcref{Harder Containment}, we see that $V(f|_{(\LG/\mathfrak{u})^*}) \subseteq \cup_{\alpha}V(h_{\alpha})$, and so $f|_{(\LG/\mathfrak{u})^*}$ is a product of some subset of the $h_{\alpha}$; moreover, this subset of $h_{\alpha}$ is nonempty since $0 \in V(f) \cap (\LG/\mathfrak{u})^*$. Choosing a divisor $h_{\alpha}$ of $f|_{(\LG/\mathfrak{u})^*}$, we therefore obtain that \[\emptyset \neq V(h_{\alpha}) \cap (\LG/\mathfrak{u})^*_{\text{reg}} \subseteq V(h_{\alpha}) \subseteq V(f|_{(\LG/\mathfrak{u})^*}) \subseteq V(f) = V.\]  In particular, $V \cap \LGd_{\mathrm{reg}}$ is nonempty, which contradicts our assumption that $V$ is a subset of $(\LG_{\mathrm{reg}}/\mathfrak{u}_P)^{*, c}_{\mathrm{reg}} \subseteq \LG_{\mathrm{reg}}^{*, c}$. 
\end{proof}

\subsection{Proof of the Main Theorem}
We have equivalences of ring objects in $\QCoh(\LGd_{\text{reg}})^G$ 
\[\mathcal{O}_{T^*(G/U_P)_{\text{reg}}} \xrightarrow{\sim} (\mathcal{O}_{G \times \mathfrak{c}} \otimes_{\mathcal{O}(\mathfrak{c})} \kappa^*(\mathcal{O}_{T^*(G/U_P)_{\text{reg}}}))^{J_G} \xrightarrow{\sim} (\mathcal{O}_{G \times \mathfrak{c}} \otimes_{\mathcal{O}(\mathfrak{c})} \mathcal{O}(T^*(\overline{U}{}_{\psi}\backslash G/U_P))^{J_G}\] \[\xrightarrow{\text{id} \otimes j^*} (\mathcal{O}_{G \times \mathfrak{c}} \otimes_{\mathcal{O}(\mathfrak{c})} \mathcal{O}(T^*(\overline{U}{}_{\psi}\backslash\bigCell))^{J_G} \xrightarrow{\text{id} \otimes h^*} (\mathcal{O}_{G \times \mathfrak{c}} \otimes_{\mathcal{O}(\mathfrak{c})} \mathcal{O}(T^*(\overline{U}^L_{\psi_L}\backslash L))^{J_G} \] where the first isomorphism is given by \cref{Any algebra object on regular locus is recovered from its Kostant-Whittaker Reduction}, the second is given by applying \cref{Scheme with Almost Hamiltonian G Action Has Easy Whittaker Reduction After Choosing Kostant Section} with $X = T^*(G/U_P)_{\mathrm{reg}}$, the third map is an isomorphism by \cref{Open Embedding of Big Cell Induces Isomorphism on Whittaker Cotangent Bundles}, and the map $h$ is an isomorphism by our above analysis after the statement of \cref{Open Embedding of Big Cell Induces Isomorphism on Whittaker Cotangent Bundles}. Using the equivalence of \cref{Scheme with Almost Hamiltonian G Action Has Easy Whittaker Reduction After Choosing Kostant Section}, we therefore obtain an isomorphism \[\mathcal{O}_{T^*(G/U_P)_{\text{reg}}} \cong (\mathcal{O}_{T^*(G/_{\psi}\overline{U})} \otimes_{\mathcal{O}(\mathfrak{c})} \mathcal{O}(T^*(\overline{U}^L_{\psi_L}\backslash L)))^{J_G}\] of quasicoherent sheaves of $\Sym(\LG \oplus \mathfrak{l})$-algebras compatible with the actions of $G$ and $L$. Taking global sections, we deduce our desired equivalence from the fact that the open embedding $\LGd_{\text{reg}} \to \LGd$ is qcqs and so $j_*$ preserves limits and \cref{Global Functions on Cotangent Bundle is Global Functions on Cotangent Bundle of G Mod UP}. This proves \cref{Restated Theorem Intro Main for Whittaker} and, as we have explained in Section \labelcref{Whit Subsection}, therefore proves \cref{thm: intro main}.

\subsection{Corollary on Partial Whittaker Cotangent Bundle}
Let $\mathring{U} := U_P\overline{U}^L$, and let $\mathring{\psi}_L: \mathring{U} \to \mathbb{G}_a$ denote the unique character whose kernel contains $U_P$ and which extends $\psi_L$. We now record the following corollary on the functions on the \textit{partial Whittaker cotangent bundle} $T^*(G/_{\mathring{\psi}_L}\mathring{U})$:

\begin{Corollary}\label{Isomorphism Induced by L Whittaker Reduction Corollary}
There is an action of $J_G$ on $J_L$ for which the induced diagonal $J_G$-action induces an isomorphism of $\Sym(\LG)$-algebras
    $$\mathcal{O}(T^*(G/_{\mathring{\psi}_L}\mathring{U})) \cong \mathcal{O}(T^*(G/_{\psi}\overline{U}) \times_{\mathfrak{c}}  J_L)^{J_G}$$ compatible with the actions of $G$ and $J_L$.
\end{Corollary}

\begin{proof}
    Our isomorphism in \cref{Restated Theorem Intro Main for Whittaker} is an isomorphism in particular compatible with the $\Sym(\mathfrak{l})$-algebra structure and the $L$-representation. We may thus apply the Kostant-Whittaker reduction functor for $L$. Since the Kostant-Whittaker reduction functor is exact by \cref{Kostant-Whittaker Reduction Is Exact Monoidal Equivalence of Categories with Explicitly Constructed Inverse}, we obtain an isomorphism \[\mathcal{O}(T^*(G/_{\mathring{\psi}_L}\mathring{U})) \cong \mathcal{O}(T^*(G/_{\psi}\overline{U}) \times_{\mathfrak{c}} T^*(\overline{U}^L_{\psi_L}\backslash L/\overline{U}_{\psi_L}^{L}))^{J_G}\] compatible with the $\Sym(\LG)$-algebra and the actions of $G$ and $J_L$. Using \cref{Scheme with Almost Hamiltonian G Action Has Easy Whittaker Reduction After Choosing Kostant Section} and the fact that the diagonal map $\mathfrak{c}_L \to \mathfrak{c}_L \times_{\mathfrak{c}_L} \mathfrak{c}_L$ is an isomorphism, we may identify $T^*(\overline{U}^L_{\psi_L}\backslash L/\overline{U}_{\psi_L}^{L})$ with the centralizer of some choice of Kostant section for $L$. By, for example, \cite[Section 2.1]{NgoLeLemmeFondamentalPourLesAlgebresdeLie}, \cite[Remark 3.3.10]{RicheKostantSectionUniversalCentralizerandModularDerivedSatakeEquivalence}, this identifies with $J_L$, as desired.
\end{proof}
\begin{Remark}\label{BZSV Connection}
For the ease of exposition, let us assume $k = \mathbb{C}$. Let $G^{\vee}$ denote the Langlands dual group to $G$, and let $L^{\vee}$ and $\mathring{U}^{\vee}$ denote the corresponding subgroups of $G^{\vee}$. \cref{Isomorphism Induced by L Whittaker Reduction Corollary} provides a natural candidate for the dual Hamiltonian $G^{\vee}$-space $M^{\vee}$ for the (not necessarily hyperspherical) Hamiltonian $G$-variety $T^*(G/L)$ in the relative Langlands duality program \cite{BenZviSakleredisVenkateshRelativeLanglandsDuality}--namely, $M^\vee$ is the affine closure of the quasi-affine variety $T^*(G^{\vee}/_{\mathring{\psi}_{L^{\vee}}}\mathring{U}^{\vee})$. See \cite[Conjecture 3.6.15]{Devalapurkar-ku-relative-Langlands} and the surrounding discussion in \textit{loc. cit} for further discussion. 
\end{Remark}

\printbibliography

@misc{Devalapurkar-ku-relative-Langlands,
      title={ku-theoretic spectral decompositions for spheres and projective spaces}, 
      author={Sanath K. Devalapurkar},
      year={2024},
      eprint={2402.03995},
      archivePrefix={arXiv},
      primaryClass={math.AT},
      url={https://arxiv.org/abs/2402.03995}, 
}

@misc{GannonWilliamsDifferentialOperatorsOnBaseAffineSpaceofSLnandQuantizedCoulombBranches,
      title={Differential operators on the base affine space of ${SL}_n$ and quantized {C}oulomb branches}, 
      author={Tom Gannon and Harold Williams},
      year={2023},
      eprint={2312.10278},
      archivePrefix={arXiv},
      primaryClass={math.RT}
}

@misc{DancerGrimmingerMartensZhongComplexSymplectic,
      title={Complex Symplectic Contractions and 3d Mirrors}, 
      author={Andrew Dancer and Julius F. Grimminger and Johan Martens and Zhenghao Zhong},
      year={2024},
      eprint={2406.09626},
      archivePrefix={arXiv},
}

@article {NgoLeLemmeFondamentalPourLesAlgebresdeLie,
    AUTHOR = {{N}g{\^{o}}, Bao Ch{\^{a}}u},
     TITLE = {Le lemme fondamental pour les alg\`ebres de {L}ie},
   JOURNAL = {Publ. Math. Inst. Hautes \'{E}tudes Sci.},
  FJOURNAL = {Publications Math\'{e}matiques. Institut de Hautes \'{E}tudes
              Scientifiques},
    NUMBER = {111},
      YEAR = {2010},
     PAGES = {1--169},
      ISSN = {0073-8301,1618-1913},
   MRCLASS = {22E35 (11S37 14D23 14G35 22E50)},
  MRNUMBER = {2653248},
MRREVIEWER = {R.\ P.\ Langlands},
       DOI = {10.1007/s10240-010-0026-7},
       URL = {https://doi.org/10.1007/s10240-010-0026-7},
}

@article {BezrukavnikovRicheAffineBraidGroupActionsonDerivedCategoriesofSpringerResolutions,
    AUTHOR = {Bezrukavnikov, Roman and Riche, Simon},
     TITLE = {Affine braid group actions on derived categories of {S}pringer
              resolutions},
   JOURNAL = {Ann. Sci. \'{E}c. Norm. Sup\'{e}r. (4)},
  FJOURNAL = {Annales Scientifiques de l'\'{E}cole Normale Sup\'{e}rieure.
              Quatri\`eme S\'{e}rie},
    VOLUME = {45},
      YEAR = {2012},
    NUMBER = {4},
     PAGES = {535--599},
      ISSN = {0012-9593,1873-2151},
   MRCLASS = {20Gxx (14F05 20F36)},
  MRNUMBER = {3059241},
       DOI = {10.24033/asens.2173},
       URL = {https://doi.org/10.24033/asens.2173},
}

@article {KostantLieGroupRepsonPolynomialRings,
    AUTHOR = {Kostant, Bertram},
     TITLE = {Lie group representations on polynomial rings},
   JOURNAL = {Amer. J. Math.},
  FJOURNAL = {American Journal of Mathematics},
    VOLUME = {85},
      YEAR = {1963},
     PAGES = {327--404},
      ISSN = {0002-9327,1080-6377},
   MRCLASS = {22.60 (20.80)},
  MRNUMBER = {158024},
MRREVIEWER = {J.\ L.\ Koszul},
       DOI = {10.2307/2373130},
       URL = {https://doi.org/10.2307/2373130},
}

@misc{MaceratoLEviEquivariantRestrictionofSphericalPerverseSheaves,
      title={Levi-Equivariant Restriction of Spherical Perverse Sheaves}, 
      author={Mark Macerato},
      year={2023},
      eprint={2309.07279},
      archivePrefix={arXiv},
      primaryClass={math.RT}
}

@article {BourgetDancerGrimmingerHananyZhongPartialImplosionsandQuivers,
    AUTHOR = {Bourget, Antoine and Dancer, Andrew and Grimminger, Julius F.
              and Hanany, Amihay and Zhong, Zhenghao},
     TITLE = {Partial implosions and quivers},
   JOURNAL = {J. High Energy Phys.},
  FJOURNAL = {Journal of High Energy Physics},
      YEAR = {2022},
    NUMBER = {7},
     PAGES = {Paper No. 49, 20},
      ISSN = {1126-6708,1029-8479},
   MRCLASS = {81T60 (14L24 81T13)},
  MRNUMBER = {4455989},
       DOI = {10.1007/jhep07(2022)049},
       URL = {https://doi.org/10.1007/jhep07(2022)049},
}

@article {KostantonWhittakerVectorsandRepresentationTheory,
    AUTHOR = {Kostant, Bertram},
     TITLE = {On {W}hittaker vectors and representation theory},
   JOURNAL = {Invent. Math.},
  FJOURNAL = {Inventiones Mathematicae},
    VOLUME = {48},
      YEAR = {1978},
    NUMBER = {2},
     PAGES = {101--184},
      ISSN = {0020-9910,1432-1297},
   MRCLASS = {22E47 (22E45)},
  MRNUMBER = {507800},
MRREVIEWER = {A.\ U.\ Klimyk},
       DOI = {10.1007/BF01390249},
       URL = {https://doi.org/10.1007/BF01390249},
}

@misc{BenZviSakleredisVenkateshRelativeLanglandsDuality,
      title={Relative Langlands Duality}, 
      author={David Ben-Zvi and Yiannis Sakellaridis and Akshay Venkatesh},
      year={2024},
      eprint={2409.04677},
      archivePrefix={arXiv},
      primaryClass={math.RT},
      url={https://arxiv.org/abs/2409.04677}, 
}

@article {GinzburgKazhdanDifferentialOperatorsOnBasicAffineSpaceandtheGelfandGraevAction,
    AUTHOR = {Ginzburg, Victor and Kazhdan, David},
     TITLE = {Differential operators on {$G / U$} and the {G}elfand-{G}raev
              action},
   JOURNAL = {Adv. Math.},
  FJOURNAL = {Advances in Mathematics},
    VOLUME = {403},
      YEAR = {2022},
     PAGES = {Paper No. 108368, 48},
      ISSN = {0001-8708},
   MRCLASS = {22E30 (16S32)},
  MRNUMBER = {4404034},
       DOI = {10.1016/j.aim.2022.108368},
       URL = {https://doi.org/10.1016/j.aim.2022.108368},
}

@book {JantzenRepresentationsofAlgebraicGroups,
    AUTHOR = {Jantzen, Jens Carsten},
     TITLE = {Representations of algebraic groups},
    SERIES = {Mathematical Surveys and Monographs},
    VOLUME = {107},
   EDITION = {Second},
 PUBLISHER = {American Mathematical Society, Providence, RI},
      YEAR = {2003},
     PAGES = {xiv+576},
      ISBN = {0-8218-3527-0},
   MRCLASS = {20G05 (17B10)},
  MRNUMBER = {2015057},
}

@article {RicheKostantSectionUniversalCentralizerandModularDerivedSatakeEquivalence,
    AUTHOR = {Riche, Simon},
     TITLE = {Kostant section, universal centralizer, and a modular derived
              {S}atake equivalence},
   JOURNAL = {Math. Z.},
  FJOURNAL = {Mathematische Zeitschrift},
    VOLUME = {286},
      YEAR = {2017},
    NUMBER = {1-2},
     PAGES = {223--261},
      ISSN = {0025-5874,1432-1823},
   MRCLASS = {17B45 (14G17 14L30 17B08)},
  MRNUMBER = {3648498},
MRREVIEWER = {Paul\ D.\ Levy},
       DOI = {10.1007/s00209-016-1761-3},
       URL = {https://doi.org/10.1007/s00209-016-1761-3},
}
\end{document}